\documentclass[11pt,a4paper]{article}
\usepackage[affil-it]{authblk}
\usepackage[utf8]{inputenc}
\usepackage{amsmath, amssymb, mathrsfs, amsthm, amsfonts}
\usepackage{latexsym}
\usepackage{hyperref}
\usepackage{a4wide}
\usepackage{color}
\usepackage{graphicx}
\usepackage{xcolor, colortbl}

\definecolor{Sky}{rgb}{0.88,1,1}

\newtheorem{theorem}{Theorem}[section]

\newtheorem{lemma}[theorem]{Lemma}
\newtheorem{proposition}[theorem]{Proposition}

\theoremstyle{definition}
\newtheorem{definition}[theorem]{Definition}
\theoremstyle{remark}

\title{The Tutte Polynomial of Complex Reflection Groups}

\author{Hery Randriamaro
\thanks{This research was funded by my mother \\
Lot II B 32 bis Faravohitra, 101 Antananarivo, Madagascar \\
e-mail: \texttt{hery.randriamaro@gmail.com}}}

\begin{document}

\maketitle

\begin{abstract}
\noindent This article computes the Tutte polynomial of the hyperplane arrangements associated to the complex reflection groups. The calculations are based on both formulas of De Concini and Procesi for Tutte polynomial and the normaliser of parabolic subgroups in complex reflection groups determined by Krishnasamy and Taylor.

\bigskip 


\smallskip

\end{abstract}

\section{Introduction}

\noindent We work in the Hermitian space $\mathbb{C}^n$ endowed with an inner product $\langle \centerdot , \centerdot \rangle: \mathbb{C}^n \times \mathbb{C}^n \rightarrow \mathbb{R}$. Denote by $\mathbb{U}_m$ the set of all $m^{\text{th}}$ roots of unity, and by $\mathbb{U}$ the set $\displaystyle \bigcup_{m \in \mathbb{N}^*} \mathbb{U}_m$. For a nonzero vector $u \in \mathbb{C}^n$ and $\xi \in \mathbb{U}$, a complex reflection is a unitary transformation $r_{u,\xi}: \mathbb{C}^n \rightarrow \mathbb{C}^n$ defined by $\displaystyle r_{u,\xi}(x) := x - (1 - \xi) \frac{\langle x,u \rangle}{\langle u,u \rangle}u$. A complex reflection group or CRG is a finite subgroup of $\mathrm{GL}(\mathbb{C}^n)$ generated by complex reflections on $\mathbb{C}^n$. CRGs play a key role in the structure as well as in the representation theory of finite reductive groups, and give rise to braid groups and generalized Hecke algebras \cite[Chapter~2, 3]{Br}. Denote by $R_G$ the set formed by the complex reflections of a CRG $G$. The hyperplane arrangement associated to $G$ is
$$\mathcal{A}_G := \big\{\ker(1-r)\ \big|\ r \in R_G\big\}.$$
The group $G$ is said irreducible if $\mathcal{A}_G$ is irreducible. It is imprimitive if, for some $k>1$, $\mathbb{C}^n$ is a direct sum of nonzero subspaces $V_1, \dots, V_k$ such that the action of $G$ on $\mathbb{C}^n$ permutes $V_1, \dots, V_k$ among themselves, otherwise it is primitive. The irreducible CRGs were classified by Shephard and Todd \cite{ShTo}. The three infinite families of irreducible CRGs are the symmetric groups $\mathrm{Sym}(n)$, the imprimitive groups $G(m,p,n)$, and the cyclic groups $C_n$. In addition there are $34$ irreducible primitive groups of ranks $2, \dots, 8$ denoted by the symbols $G_4, G_5, \dots, G_{37}$. Each irreducible CRG has minimum sets of complex reflections generating it and subject to braid relations \cite[Appendix~A]{Br2}. A parabolic subgroup of a CRG $G$ in $\mathbb{C}^n$ is the pointwise stabilizer of a subspace of $\mathbb{C}^n$. Steinberg proved that a such subgroup is also a CRG \cite[Theorem~1.5]{St}. Furthermore, Krishnasamy and Taylor determined the normalisers of parabolic subgroups in an irreducible CRG \cite{KrTa}. 

\noindent Recall that the rank of a hyperplane arrangement $\mathcal{A}$ in $\mathbb{C}^n$ is $\displaystyle \mathrm{rk}\,\mathcal{A} := n - \dim \bigcap_{H \in \mathcal{A}}H$.

\begin{definition}
Let $G$ be a complex reflection group, and $x,y$ two variables. The Tutte polynomial $T_G(x,y)$ associated to $G$ is the Tutte polynomial of $\mathcal{A}_G$, that is $$T_G(x,y) := \sum_{\mathcal{B} \subseteq \mathcal{A}_G} (x-1)^{\mathrm{rk}\,\mathcal{A}_G - \mathrm{rk}\,\mathcal{B}} (y-1)^{\#\mathcal{B} - \mathrm{rk}\,\mathcal{B}}.$$
\end{definition}

\noindent A graph coloring corresponds to a way of coloring so that two connected vertices are differently colored. The chromatic polynomial is a graph polynomial which counts the number of graph colorings. In 1954, Tutte obtained a polynomial from which the chromatic polynomial of a graph, and that of its dual graph can be deduced \cite[§~3]{Tu}: it is originally the Tutte polynomial. That polynomial reveals more of the internal structure of the graph like its number of forests, of spanning subgraphs, and of acyclic orientations. But beyond graphs, it has many applications as stated by its instigator in one of his last works \cite{Tu2}: “Later I was astonished to hear that it had found applications in other branches of mathematics, even in knot theory.” For any positive integer $q$ for instance, along the hyperbola $(x-1)(y-1)=q$, the Tutte polynomial specializes to the partition function of the $q$-state Potts model \cite[§~I]{MeWe}. The Tutte polynomial is also defined on other combinatorial objects like matroids \cite{MeRaRo}. But this article is predominately interested in its definition on hyperplane arrangements. Orlik and Solomon proved that the Poincaré polynomial of the cohomology ring of $\displaystyle M_G = \mathbb{C}^n \setminus \bigcup_{H \in \mathcal{A}_G} H$ is given by \cite[Theorem~5]{OrSo} $\displaystyle \sum_{k \in \mathbb{N}} \mathrm{rank}\, H^k(M_G, \mathbb{Z})\, q^k = (-1)^{\mathrm{rk}\,\mathcal{A}_G} q^{n- \mathrm{rk}\,\mathcal{A}_G} T_G(1-q,0)$.

\noindent One trivial case is that of the cyclic group $C_n$ of rank $1$ for which $\mathcal{A}_{C_n} = \{0\}$, and then $T_{C_n}(x,y) = x$. For every CRG $G$, there exist some irreducible CRGs $G^{(1)}, \dots, G^{(m)}$ such that $G \simeq G^{(1)} \times \dots \times G^{(m)}$ \cite[Theorem~1.27]{LeTa}, and then $\displaystyle T_G(x,y) = \prod_{i \in [m]} T_{G^{(i)}}(x,y)$. Namely, the Tutte polynomial associated to a CRG can be computed from those of irreducible ones.

\noindent The story "Tutte Polynomial of Reflection Group" begins in 2007 when Ardila computed the Tutte polynomial of the hyperplane arrangements associated to the symmetric groups $\mathrm{Sym}(n)$, and to the imprimitive groups $G(2,1,n)$ and $G(2,2,n)$ \cite[Theorem~4.1--4.3]{Ar} by means of the finite field method. One year later, De Concini and Procesi obtained the same polynomials with a more direct method \cite[§~3.3]{DePr}, and computed also the Tutte polynomial associated to the primitive groups $G_{28}, G_{35}, G_{36}, G_{37}$ \cite[§~3.4]{DePr}. Independently, the PhD thesis of Geldon defended in 2009 consists on the computing of those three latter polynomials \cite{Ge}. Then in 2017, we computed the Tutte polynomial associated to the imprimitive groups $G(m,p,n)$ and $G(m,m,n)$ by means of an extended field method \cite[§~5]{Ra}.

\smallskip

\noindent Let $\mathrm{Cl}_G(X)$ and $\mathrm{N}_G(X)$ be the conjugacy class and normaliser respectively of a subset $X$ in a CRG $G$. Denote by $\mathscr{C}(G)$ be the set formed by the conjugacy classes of the parabolic subgroups of $G$. This article aims to compute the Tutte polynomial associated to the imprimitive CRGs with a more direct method using the proof strategy of De Concini and Procesi, and to close the chapter relating to the above mentioned story on the complex reflection groups by computing the Tutte polynomial associated to the primitive CRGs $G_4, \dots, G_{27}, G_{29}, \dots, G_{34}$ through the following theorems.

\begin{theorem} \label{ThImp}
Let $m,p \in \mathbb{N}^*$ such that $m \neq p$ and $p \mid m$. The exponential generating function of the Tutte polynomials associated to the imprimitive complex reflection groups $G(m,p,n)$ is
$$\displaystyle \sum_{n \in \mathbb{N}} \frac{T_{G(m,p,n)}(x,y)}{n!} t^n = \bigg(\sum_{n \in \mathbb{N}} \frac{y^{m\binom{n}{2}+n}}{(y-1)^nn!}t^n\bigg) \bigg(\sum_{n \in \mathbb{N}} \frac{m^n y^{\binom{n}{2}}}{(y-1)^n n!}t^n\bigg)^{\frac{(x-1)(y-1)-1}{m}}.$$
Besides, resetting $T_{G(m,m,n)}(x,y) = x-1$ by abuse of notation, that of the Tutte polynomials associated to the imprimitive complex reflection groups $G(m,m,n)$ is 
$$\sum_{n \in \mathbb{N}} \frac{T_{G(m,m,n)}(x,y)}{n!} t^n = \bigg(\sum_{n \in \mathbb{N}} \frac{y^{m\binom{n}{2}}}{(y-1)^nn!}t^n\bigg) \bigg(\sum_{n \in \mathbb{N}} \frac{m^n y^{\binom{n}{2}}}{(y-1)^n n!}t^n\bigg)^{\frac{(x-1)(y-1)-1}{m}}.$$
\end{theorem}

\begin{theorem} \label{ThPri}
Let $G$ be a primitive complex reflection group. We obtain the Tutte polynomials associated to the groups $G_4, \dots, G_{22}$ of rank $2$ with the formula
\begin{equation}
T_G(x,y) = x^2 + (\#\mathcal{A}_G-2)x + \sum_{i \in [\#\mathcal{A}_G-1] \setminus \{1\}} (\#\mathcal{A}_G-i) y^{i-1}.  \label{Eq1}
\end{equation}
Then, we obtain the Tutte polynomials associated to the primitive groups $G_{23}, \dots, G_{27}, G_{29}, \dots, G_{34}$ of ranks $3,4,5,6$ with the recurrence relation 
\begin{equation}
T_{G}(x,y) = y^{\#\mathcal{A}_G} + \sum_{\mathrm{Cl}_G(P) \in \mathscr{C}(G) \setminus \mathrm{Cl}_G(G)} \big[G : \mathrm{N}_G(P)\big] \big((x-1)^{\mathrm{rk}\,\mathcal{A}_G - \mathrm{rk}\,\mathcal{A}_P} - (y-1)^{\mathrm{rk}\,\mathcal{A}_P}\big) T_{P}(1,y). \label{Eq2}
\end{equation}
\end{theorem}

\noindent Remark that $y^{\#\mathcal{A}_G}$ appears in Equation~\ref{Eq2} even if $\#\mathcal{A}_G - \mathrm{rk}\,\mathcal{A}_G$ is the highest power of $y$ in $T_{G}(x,y)$. That appearance will be more understandable after having read the proof of Equation~\ref{Eq2}. Note that one can also compute the Tutte polynomials associated to $G_{28}, G_{35}, G_{36}, G_{37}$ by using Equation~\ref{Eq2}. Furthermore, although the Tutte polynomial is in principle calculable from its definition, in practice that may be very cumbersome. Already for the CRG $G_{30} = H_4$, it takes a wide amount of computer time and space. 

\smallskip

\noindent This article is organized as follows: We first consider the imprimitive CRGs by proving Theorem~\ref{ThImp} in Section~\ref{SecImp}. Then, we recall the formula of Crapo in Section~\ref{SecRk2}, and use it to prove Equation~\ref{Eq1} in order to obtain the Tutte polynomials associated to the primitive CRGs of rank $2$. We prove Equation~\ref{Eq2} in Section~\ref{SecRk3}, and use it to compute the Tutte polynomials associated to the primitive CRGs of rank $3$, of rank $4$ in Section~\ref{SecRk4}, and of ranks $5$ and $6$ in Section~\ref{SecRk5}. That last computing finishes the proof of Theorem~\ref{ThPri}. The calculations are implemented with the computer algebra system \href{http://www.sagemath.org/}{SageMath}.

\section{The Imprimitive Complex Reflection Groups} \label{SecImp}

\noindent We recall two formulas of De Concini and Procesi, and use them to compute the exponential generating functions of the Tutte polynomials associated to the imprimitive reflection groups.

\smallskip

\noindent Let $\mathcal{A}$ be a hyperplane arrangement in $\mathbb{C}^n$. The closure of a subset $\mathcal{B} \subseteq \mathcal{A}$ in $\mathcal{A}$ is $$\bar{\mathcal{B}} := \big\{H \in \mathcal{A}\ \big|\ \mathrm{rk}(\mathcal{B} \cup \{H\}) = \mathrm{rk}\,\mathcal{B}\big\}.$$
The subset $\mathcal{B}$ is a flat of $\mathcal{A}$ if $\bar{\mathcal{B}} = \mathcal{B}$. Denote by $\mathrm{F}(\mathcal{A})$ the set formed by the flats of $\mathcal{A}$. We need the following result due to De Concini and Procesi \cite[Proposition~2.36]{DePr2}.

\begin{proposition} \label{PrDePr}
	Let $\mathcal{A}$ be a hyperplane arrangement in $\mathbb{C}^n$. Then,
	\begin{align}
	T_{\mathcal{A}}(x,y) & = \sum_{\mathcal{B} \in \mathrm{F}(\mathcal{A})} (x-1)^{\mathrm{rk}\,\mathcal{A} - \mathrm{rk}\,\mathcal{B}}\, T_{\mathcal{B}}(1,y), \label{EqDP1} \\
	y^{\#\mathcal{A}} & = \sum_{\mathcal{B} \in \mathrm{F}(\mathcal{A})} (y-1)^{\mathrm{rk}\,\mathcal{B}}\, T_{\mathcal{B}}(1,y). \label{EqDP2}
	\end{align}
\end{proposition}

\noindent Denote by $\mathbf{P}(G)$ the set formed by the parabolic subgroups of a complex reflection group $G$.

\begin{lemma}  \label{LeSt}
	Let $G$ be a complex reflection group. Then, there is a one-to-one correspondence between the parabolic subgroups in $\mathbf{P}(G)$ and the flats in $\mathrm{F}(\mathcal{A}_G)$ so that, if $P \in \mathbf{P}(G)$, then its corresponding flat is $\mathcal{A}_P = \big\{\ker(1-r)\ \big|\ r \in R_P\big\}$.
\end{lemma}

\begin{proof}
	Assume that $\mathcal{A}_G$ is a hyperplane arrangement in $\mathbb{C}^n$. Then, $P$ is a parabolic subgroup of $G$ \: if and only if \: there exists a subspace $\displaystyle V \subseteq \mathbb{C}^n$ such that $\mathcal{A}_P = \{H \in \mathcal{A}_G\ |\ V \subseteq H\}$ and $\displaystyle \bigcap_{H \in \mathcal{A}_P} H = V$ \: if and only if \: $\mathcal{A}_P$ is a flat of $\mathcal{A}_G$.
\end{proof}

\noindent We can now proceed to the proof of Theorem~\ref{ThImp}:

\begin{proof}
A subgroup $G$ of $G(m,p,n)$ is parabolic if and only if a partition $1^{h_1}2^{h_2} \dots (n-k)^{h_{n-k}}$ with $\displaystyle \sum_{i \in [n-k]}ih_i = n-k$ exists such that $\displaystyle G \simeq G(m,p,k) \times \prod_{i \in [n-k]} \mathrm{Sym}(i)^{h_i}$ \cite[Theorem~3.6]{KrTa}. The number of parabolic subgroups of that type is
$\displaystyle \frac{m^{n-k-\sum_{i \in [n-k]}h_i} n!}{k! \prod_{i \in [n-k]}i!^{h_i}h_i!}$ \cite[Lemma~3.5]{KrTa}.

\noindent Let $\mathrm{a}_n(y) := T_{\mathrm{Sym}(n)}(1,y)$, $\mathrm{b}_n(y) := T_{G(m,p,n)}(1,y)$, $\mathrm{d}_n(y) := T_{G(m,m,n)}(1,y)$, and assume $\mathrm{a}_1(y) = \mathrm{b}_0(y) = \mathrm{d}_0(y) = 1$. As $\displaystyle \#\mathcal{A}_{G(m,p,n)} = m\binom{n}{2}+n$ and $\displaystyle \#\mathcal{A}_{G(m,m,n)} = m\binom{n}{2}$, using Equation~\ref{EqDP2} we obtain

\begin{align*}
& G(m,p,n): & \frac{y^{m\binom{n}{2}+n}}{m^nn!} = \sum_{\substack{k, h_1, \dots, h_{n-k} \in \mathbb{N} \\ \sum_{i \in [n-k]}ih_i = n-k}} (y-1)^{k+\sum_{i \in [n-k]}(i-1)h_i} \frac{\mathrm{b}_k(y)}{m^kk!} \prod_{i \in [n-k]}\frac{\mathrm{a}_i(y)^{h_i}}{(mi!)^{h_i}h_i!}, \\
& G(m,m,n): & \frac{y^{m\binom{n}{2}}}{m^nn!} = \sum_{\substack{k, h_1, \dots, h_{n-k} \in \mathbb{N} \\ k \neq 1 \\ \sum_{i \in [n-k]}ih_i = n-k}} (y-1)^{k+\sum_{i \in [n-k]}(i-1)h_i} \frac{\mathrm{d}_k(y)}{m^kk!} \prod_{i \in [n-k]}\frac{\mathrm{a}_i(y)^{h_i}}{(mi!)^{h_i}h_i!}.
\end{align*}

\noindent In term of generating functions, we have
\begin{align*}
\sum_{n \in \mathbb{N}} \frac{y^{m\binom{n}{2}+n}}{m^nn!}t^n & = \sum_{n \in \mathbb{N}} \sum_{\substack{k, h_1, \dots, h_{n-k} \in \mathbb{N} \\ \sum_{i \in [n-k]}ih_i = n-k}} \frac{(y-1)^k\mathrm{b}_k(y)t^k}{m^kk!} \prod_{i \in [n-k]}\frac{\big((y-1)^{i-1}\mathrm{a}_i(y)t^i\big)^{h_i}}{(mi!)^{h_i}h_i!} \\
& = \bigg(\sum_{k \in \mathbb{N}} \frac{\mathrm{b}_k(y)}{m^kk!} \big((y-1)t\big)^k \bigg) \, \exp \sum_{i \in \mathbb{N}^*} \frac{(y-1)^{i-1}\mathrm{a}_i(y)}{mi!} t^i \\
& = \bigg(\sum_{k \in \mathbb{N}} \frac{\mathrm{b}_k(y)}{m^kk!} \big((y-1)t\big)^k \bigg) \bigg(\exp \sum_{i \in \mathbb{N}^*} \frac{(y-1)^{i-1}\mathrm{a}_i(y)}{i!} t^i\bigg)^{\frac{1}{m}} \\
& = \bigg(\sum_{k \in \mathbb{N}} \frac{\mathrm{b}_k(y)}{m^kk!} \big((y-1)t\big)^k \bigg) \bigg(\sum_{n \in \mathbb{N}}\frac{y^{\binom{n}{2}}}{n!}t^n\bigg)^{\frac{1}{m}}\ \text{from \cite[Equation~2.24]{DePr2}},
\end{align*}
and also $\displaystyle \sum_{n \in \mathbb{N}} \frac{y^{m\binom{n}{2}}}{m^nn!}t^n = \bigg(\sum_{k \in \mathbb{N} \setminus \{1\}} \frac{\mathrm{d}_k(y)}{m^kk!} \big((y-1)t\big)^k \bigg) \bigg(\sum_{n \in \mathbb{N}}\frac{y^{\binom{n}{2}}}{n!}t^n\bigg)^{\frac{1}{m}}$. Then,
\begin{align}
\sum_{k \in \mathbb{N}} \frac{\mathrm{b}_k(y)}{m^kk!} \big((y-1)t\big)^k & = \bigg(\sum_{n \in \mathbb{N}} \frac{y^{m\binom{n}{2}+n}}{m^nn!}t^n\bigg) \bigg(\sum_{n \in \mathbb{N}}\frac{y^{\binom{n}{2}}}{n!}t^n\bigg)^{-\frac{1}{m}}, \label{EqBk} \\
\sum_{k \in \mathbb{N} \setminus \{1\}} \frac{\mathrm{d}_k(y)}{m^kk!} \big((y-1)t\big)^k & = \bigg(\sum_{n \in \mathbb{N}} \frac{y^{m\binom{n}{2}}}{m^nn!}t^n\bigg) \bigg(\sum_{n \in \mathbb{N}}\frac{y^{\binom{n}{2}}}{n!}t^n\bigg)^{-\frac{1}{m}}. \label{EqDk}
\end{align}

\noindent Now for the Tutte polynomials, using Equation~\ref{EqDP1} we have
\begin{align*}
& \frac{T_{G(m,p,n)}(x,y)}{m^n n!} = \sum_{\substack{k, h_1, \dots, h_{n-k} \in \mathbb{N} \\ \sum_{i \in [n-k]}ih_i = n-k}} (x-1)^{\sum_{i \in [n-k]} h_i} \frac{\mathrm{b}_k(y)}{m^kk!} \prod_{i \in [n-k]}\frac{\mathrm{a}_i(y)^{h_i}}{(mi!)^{h_i}h_i!}, \\
& \frac{T_{G(m,m,n)}(x,y)}{m^n n!} = \sum_{\substack{k, h_1, \dots, h_{n-k} \in \mathbb{N} \\ k \neq 1 \\ \sum_{i \in [n-k]}ih_i = n-k}} (x-1)^{\sum_{i \in [n-k]} h_i} \frac{\mathrm{d}_k(y)}{m^kk!} \prod_{i \in [n-k]}\frac{\mathrm{a}_i(y)^{h_i}}{(mi!)^{h_i}h_i!},
\end{align*}
where $T_{G(m,m,n)}(x,y) = x-1$ by abuse of notation. In term of generating functions, we get
\begin{align*}
\sum_{n \in \mathbb{N}} \frac{T_{G(m,p,n)}(x,y)}{m^n n!} t^n & = \sum_{n \in \mathbb{N}} \sum_{\substack{k, h_1, \dots, h_{n-k} \in \mathbb{N} \\ \sum_{i \in [n-k]}ih_i = n-k}} \frac{\mathrm{b}_k(y)t^k}{m^kk!} \prod_{i \in [n-k]}\frac{\big((x-1)\mathrm{a}_i(y)t^i\big)^{h_i}}{(mi!)^{h_i}h_i!} \\
& = \bigg(\sum_{k \in \mathbb{N}} \frac{\mathrm{b}_k(y)}{m^kk!} t^k\bigg) \, \exp \sum_{i \in \mathbb{N}^*} \frac{(x-1)\mathrm{a}_i(y)}{mi!} t^i \\
& = \bigg(\sum_{k \in \mathbb{N}} \frac{\mathrm{b}_k(y)}{m^kk!} t^k \bigg) \bigg(\exp \sum_{i \in \mathbb{N}^*} \frac{\mathrm{a}_i(y)}{i!} t^i\bigg)^{\frac{x-1}{m}}.
\end{align*}

\noindent Using Equation~\ref{EqBk}, we obtain $\displaystyle \sum_{k \in \mathbb{N}} \frac{\mathrm{b}_k(y)}{m^kk!} t^k = \bigg(\sum_{n \in \mathbb{N}} \frac{y^{m\binom{n}{2}+n}}{\big(m(y-1)\big)^nn!}t^n\bigg) \bigg(\sum_{n \in \mathbb{N}}\frac{y^{\binom{n}{2}}}{(y-1)^nn!}t^n\bigg)^{-\frac{1}{m}}$. Besides,
\begin{align*}
\exp \sum_{i \in \mathbb{N}^*} \frac{\mathrm{a}_i(y)}{i!} t^i & = \exp \sum_{n \in \mathbb{N}^*} \frac{(y-1)^n\mathrm{a}_n(y)}{n!} \Big(\frac{t}{y-1}\Big)^n \\
& = \bigg(\exp \sum_{n \in \mathbb{N}^*} \frac{(y-1)^{n-1}\mathrm{a}_n(y)}{n!} \Big(\frac{t}{y-1}\Big)^n\bigg)^{y-1} \\
& = \bigg(\sum_{n \in \mathbb{N}} \frac{y^{\binom{n}{2}}}{(y-1)^n n!}t^n\bigg)^{y-1}\ \text{from \cite[Equation~2.24]{DePr2}}
\end{align*}

\noindent Hence $\displaystyle \sum_{n \in \mathbb{N}} \frac{T_{G(m,p,n)}(x,y)}{m^n n!} t^n = \bigg(\sum_{n \in \mathbb{N}} \frac{y^{m\binom{n}{2}+n}}{\big(m(y-1)\big)^nn!}t^n\bigg) \bigg(\sum_{n \in \mathbb{N}} \frac{y^{\binom{n}{2}}}{(y-1)^n n!}t^n\bigg)^{\frac{(x-1)(y-1)-1}{m}}$.

\noindent Likewise, using Equation~\ref{EqDk} we obtain
\begin{align*}
\sum_{n \in \mathbb{N}} \frac{T_{G(m,m,n)}(x,y)}{m^n n!} t^n & = \sum_{n \in \mathbb{N}} \sum_{\substack{k, h_1, \dots, h_{n-k} \in \mathbb{N} \\ k \neq 1 \\ \sum_{i \in [n-k]}ih_i = n-k}} \frac{\mathrm{d}_k(y)t^k}{m^kk!} \prod_{i \in [n-k]}\frac{\big((x-1)\mathrm{a}_i(y)t^i\big)^{h_i}}{(mi!)^{h_i}h_i!} \\
& = \bigg(\sum_{k \in \mathbb{N} \setminus \{1\}} \frac{\mathrm{d}_k(y)}{m^kk!} t^k\bigg) \, \exp \sum_{i \in \mathbb{N}^*} \frac{(x-1)\mathrm{a}_i(y)}{mi!} t^i \\
& = \bigg(\sum_{k \in \mathbb{N} \setminus \{1\}} \frac{\mathrm{d}_k(y)}{m^kk!} t^k \bigg) \bigg(\sum_{n \in \mathbb{N}} \frac{y^{\binom{n}{2}}}{(y-1)^n n!}t^n\bigg)^{\frac{(x-1)(y-1)}{m}} \\
& = \bigg(\sum_{n \in \mathbb{N}} \frac{y^{m\binom{n}{2}}}{\big(m(y-1)\big)^nn!}t^n\bigg) \bigg(\sum_{n \in \mathbb{N}} \frac{y^{\binom{n}{2}}}{(y-1)^n n!}t^n\bigg)^{\frac{(x-1)(y-1)-1}{m}}. 
\end{align*}
\end{proof}

\section{The Primitive Complex Reflection Groups of Rank $2$} \label{SecRk2}

\noindent We first expose the formula of Crapo. The first reason is we use it to prove Equation~\ref{Eq1} of Theorem~\ref{ThPri}. The second is we implement it to compute intermediate Tutte polynomials like $T_{\mathrm{Sym}(5)}(x,y), T_{G(2,2,4)}(x,y), T_{G(3,3,4)}(x,y)$ to obtain $T_{K_5}(x,y)$ for example. It indeed has the advantage to reduce the implementation on $\displaystyle \binom{\#\mathcal{A}_G}{\mathrm{rk}\,\mathcal{A}_G}$ sets instead of $2^{\#\mathcal{A}_G}$. Then, we describe how to obtain the Tutte polynomials associated to the primitive CRGs of rank $2$.

\smallskip

\noindent Let $\mathcal{A}$ be a hyperplane arrangement in $\mathbb{C}^n$. A basis of $\mathcal{A}$ is a subset $\mathcal{B} \subseteq \mathcal{A}$ such that $$\#\mathcal{B} = \mathrm{rk}\,\mathcal{A} \quad \text{and} \quad \mathrm{rk}\,\mathcal{B} = \mathrm{rk}\,\mathcal{A}.$$

\noindent Denote by $\mathrm{B}(\mathcal{A})$ the set formed by the basis of $\mathcal{A}$. Moreover if $\mathcal{A}$ has a linear order $\lhd$, for $\mathcal{B} \subseteq \mathcal{A}$ and $H \in \mathcal{A}$, define the set $\mathcal{B}_{\lhd H} := \{K \in \mathcal{B}\ |\ K \lhd H\}$.

\smallskip

\noindent Let $\mathcal{A}$ be a hyperplane arrangement in $\mathbb{C}^n$ with a linear order $\lhd$, and $\mathcal{B} \in \mathrm{B}(\mathcal{A})$:
\begin{itemize}
\item Let $K \in \mathcal{B}$. One says that $K$ is an internal active element of $\mathcal{B}$ if $$\forall H \in \mathcal{A}_{\lhd K} \setminus \mathcal{B}:\, \mathrm{rk}\big(\{H\} \sqcup (\mathcal{B} \setminus \{K\})\big) < \mathrm{rk}\,\mathcal{A}.$$
\item Let $H \in \mathcal{A} \setminus \mathcal{B}$. One says that $H$ is an external active element of $\mathcal{B}$ if
$$\mathrm{rk}\big(\{H\} \sqcup \mathcal{B}_{\rhd H}\big) = \mathrm{rk}(\mathcal{B}_{\rhd H}).$$
\end{itemize}

\noindent Denote by $\mathrm{I}(\mathcal{B})$ resp. $\mathrm{E}(\mathcal{B})$ the set of internal resp. external active elements of a basis $\mathcal{B}$. We can now state the formula of Crapo \cite[Theorem~2.32]{DePr2}.

\begin{theorem}
Let $\mathcal{A}$ be a hyperplane arrangement in $\mathbb{C}^n$ with a linear order. Then, the Tutte polynomial of $\mathcal{A}$ is
$$T_{\mathcal{A}}(x,y) = \sum_{\mathcal{B} \in \mathrm{B}(\mathcal{A})} x^{\#\mathrm{I}(\mathcal{B})} y^{\#\mathrm{E}(\mathcal{B})}.$$
\end{theorem}

\noindent Now, let $G$ be a CRG of rank $2$, and define the linear order $\prec$ on $\mathcal{A}_G = \{H_i\}_{i \in [m]}$, for $H_i, H_j \in \mathcal{A}_G$, by: $H_i \prec H_j \Longleftrightarrow i < j$. Hence,
\begin{itemize}
\item $\mathrm{I}\big(\{H_1, H_2\}\big) = \{H_1, H_2\}$ and $\mathrm{E}\big(\{H_1, H_2\}\big) = \emptyset$,
\item if $2 < j \leq m$, then $\mathrm{I}\big(\{H_1, H_j\}\big) = \{H_1\}$ and $\mathrm{E}\big(\{H_1, H_j\}\big) = \emptyset$,
\item if $1 < i < j \leq m$, then $\mathrm{I}\big(\{H_i, H_j\}\big) = \emptyset$ and $\mathrm{E}\big(\{H_i, H_j\}\big) = \big\{H_k\ \big|\ k \in [i-1]\big\}$.
\end{itemize}
Therefore, 
\begin{align*}
T_G(x,y) & = \sum_{i \in [m-1]} \sum_{j \in [m] \setminus [i]} x^{\#\mathrm{I}(\{H_i, H_j\})} y^{\#\mathrm{E}(\{H_i, H_j\})} \\
& = x^2 + \sum_{j \in [m] \setminus \{1,2\}} x + \sum_{i \in [m-1] \setminus \{1\}} \sum_{j \in [m] \setminus [i]} y^{i-1} \\
& = x^2 + (m-2)x + \sum_{i \in [m-1] \setminus \{1\}} (m-i) y^{i-1} \\
& = x^2 + (\#\mathcal{A}_G-2)x + \sum_{i \in [\#\mathcal{A}_G-1] \setminus \{1\}} (\#\mathcal{A}_G-i) y^{i-1}.
\end{align*}

\noindent We obtain the Tutte polynomials associated to the CRGs $G_4, G_5, \dots, G_{22}$ by replacing $\#\mathcal{A}_G$ to the corresponding cardinalities $\#\mathcal{A}_{G_k}$ listed in Table \ref{Rk2}. Those cardinalities were obtained from \cite[Table~3]{Co}. $T$ denotes the binary tetrahedral group of order $24$, $O$ the binary octahedral group of order $48$, and $I$ the binary icosahedral group of order $120$. Besides, the symbol $A \circ B$ denotes the central product of subgroups $A$ and $B$.

\begin{table}
\begin{center}
\begin{tabular}{|>{\columncolor{Sky}} c | c | c ||>{\columncolor{Sky}} c | c | c ||>{\columncolor{Sky}} c | c | c |}
  \hline			
  $k$ & $G_k$ & $\#\mathcal{A}_{G_k}$ & $k$ & $G_k$ & $\#\mathcal{A}_{G_k}$ & $k$ & $G_k$ & $\#\mathcal{A}_{G_k}$ \\
  \hline
  4 & $SL_2(\mathbb{F}_3)$ & $4$ & 11 & $C_3 \times (C_8 \circ O)$ & $46$ & 18 & $C_{15} \times I$ & $32$ \\
  5 & $C_3 \times T$ & $8$ & 12 & $GL_2(\mathbb{F}_3)$ & $12$ & 19 & $C_{15} \times (C_4 \circ I)$ & $62$ \\
  6 & $C_4 \circ SL_2(\mathbb{F}_3)$ & $10$ & 13 & $C_4 \circ O$ & $18$ & 20 & $C_3 \times I$ & $20$ \\
  7 & $C_3 \times (C_4 \circ T)$ & $14$ & 14 & $C_3 \circ GL_2(\mathbb{F}_3)$ & $20$ & 21 & $C_3 \times (C_4 \circ I)$ & $50$ \\
  8 & $T\,C_4$ & $18$ & 15 & $C_3 \times (C_4 \circ O)$ & $26$ & 22 & $C_4 \times I$ & $30$ \\
  9 & $C_8 \circ O$ & $30$ & 16 & $C_5 \times I$ & $12$ &  &  &  \\
  10 & $C_3 \times T\,C_4$ & $34$ & 17 & $C_5 \times (C_4 \circ I)$ & $42$ &  &  &  \\
  \hline  
\end{tabular}
\end{center}
\caption{The Irreducible Complex Reflection Groups of Rank $2$} \label{Rk2}
\end{table}

\section{The Primitive Complex Reflection Groups of Rank $3$}  \label{SecRk3}

\noindent We first prove Equation~\ref{Eq2} of Theorem~\ref{ThPri}. Then, we use it to compute the Tutte polynomials associated to the primitive CRGs of rank $3$. Recall that we also use it to compute those associated to the primitive CRGs of higher rank in Section~\ref{SecRk4} and Section~\ref{SecRk5}.

\begin{proof}
Using Proposition~\ref{PrDePr}, we get
\begin{align*}
T_G(x,y) & = \sum_{\mathcal{A} \in \mathrm{F}(\mathcal{A}_G)} (x-1)^{\mathrm{rk}\,\mathcal{A}_G - \mathrm{rk}\,\mathcal{A}}\, T_{\mathcal{A}}(1,y) \\
& = \sum_{\mathcal{A} \in \mathrm{F}(\mathcal{A}_G) \setminus {\mathcal{A}_G}} (x-1)^{\mathrm{rk}\,\mathcal{A}_G - \mathrm{rk}\,\mathcal{A}}\, T_{\mathcal{A}}(1,y) + T_G(1,y) \\
& = \sum_{\mathcal{A} \in \mathrm{F}(\mathcal{A}_G) \setminus {\mathcal{A}_G}} (x-1)^{\mathrm{rk}\,\mathcal{A}_G - \mathrm{rk}\,\mathcal{A}}\, T_{\mathcal{A}}(1,y) + y^{\#\mathcal{A}_G} - \sum_{\mathcal{A} \in \mathrm{F}(\mathcal{A}_G) \setminus {\mathcal{A}_G}} (y-1)^{\mathrm{rk}\,\mathcal{A}}\, T_{\mathcal{A}}(1,y) \\
& = y^{\#\mathcal{A}_G} + \sum_{\mathcal{A} \in \mathrm{F}(\mathcal{A}_G) \setminus {\mathcal{A}_G}} \big((x-1)^{\mathrm{rk}\,\mathcal{A}_G - \mathrm{rk}\,\mathcal{A}} - (y-1)^{\mathrm{rk}\,\mathcal{A}}\big) T_{\mathcal{A}}(1,y).
\end{align*}
From Lemma~\ref{LeSt}, we get 
\begin{align*}
T_G(x,y) & = y^{\#\mathcal{A}_G} + \sum_{P \in \mathbf{P}(G) \setminus {G}} \big((x-1)^{\mathrm{rk}\,\mathcal{A}_G - \mathrm{rk}\,\mathcal{A}_P} - (y-1)^{\mathrm{rk}\,\mathcal{A}_P}\big) T_{\mathcal{A}_P}(1,y) \\
& = y^{\#\mathcal{A}_G} + \sum_{\mathrm{Cl}_G(P) \in \mathscr{C}(G) \setminus \mathrm{Cl}_G(G)} \#\mathrm{Cl}_G(P) \big((x-1)^{\mathrm{rk}\,\mathcal{A}_G - \mathrm{rk}\,\mathcal{A}_P} - (y-1)^{\mathrm{rk}\,\mathcal{A}_P}\big) T_{P}(1,y).
\end{align*}
It is known that for a set $X \subseteq G$, we have $\#\mathrm{Cl}_G(X) = \big[G : \mathrm{N}_G(X)\big]$. Hence,
$$T_G(x,y) = y^{\#\mathcal{A}_G} + \sum_{\mathrm{Cl}_G(P) \in \mathscr{C}(G) \setminus \mathrm{Cl}_G(G)} \big[G : \mathrm{N}_G(P)\big] \big((x-1)^{\mathrm{rk}\,\mathcal{A}_G - \mathrm{rk}\,\mathcal{A}_P} - (y-1)^{\mathrm{rk}\,\mathcal{A}_P}\big) T_{P}(1,y).$$
\end{proof}

\noindent We can now compute the Tutte polynomials associated to the primitive CRGs $H_3$, $J_3^{(4)}$, $L_3$, $M_3$, $J_3^{(5)}$. The calculations are done using the cardinalities of conjugacy classes in Table~\ref{Rk3} which is established by means of $\#\mathrm{Cl}_G(X) = \big[G : \mathrm{N}_G(X)\big]$ and \cite[Table~4]{KrTa}.

\begin{table}
\begin{center}
\begin{tabular}{| c |>{\columncolor{Sky}} c | c |}
  \hline			
   & $P$ & $\#\mathrm{Cl}_{H_3}(P)$  \\
  \hline
  $G_{23} = H_3$ & $\mathrm{Sym}(2)$ & $15$  \\
   & $2\mathrm{Sym}(2)$ & $15$  \\
   & $\mathrm{Sym}(3)$ & $10$  \\
   & $G(5,5,2)$ & $6$  \\
  \hline			
   &  & $\#\mathrm{Cl}_{J_3^{(4)}}(P)$  \\
  \hline
  $G_{24} = J_3^{(4)}$ & $\mathrm{Sym}(2)$ & $21$  \\
   & $\mathrm{Sym}(3)$ & $28$  \\
   & $G(2,1,2)$ & $21$  \\
  \hline 
   &  & $\#\mathrm{Cl}_{L_3}(P)$  \\
  \hline
  $G_{25} = L_3$ & $C_3$ & $12$  \\
   & $2C_3$ & $12$   \\
   & $SL_2(\mathbb{F}_3)$ & $9$  \\
  \hline
   &  & $\#\mathrm{Cl}_{M_3}(P)$  \\
  \hline
  $G_{26} = M_3$ & $\mathrm{Sym}(2)$ & $9$  \\
   & $C_3$ & $12$  \\
   & $\mathrm{Sym}(2) + C_3$ & $36$  \\
   & $SL_2(\mathbb{F}_3)$ & $9$  \\
   & $G(3,1,2)$ & $12$   \\
  \hline
   &  & $\#\mathrm{Cl}_{J_3^{(5)}}(P)$  \\
  \hline
  $G_{27} = J_3^{(5)}$ & $\mathrm{Sym}(2)$ & $45$  \\
   & $\mathrm{Sym}(3)$ & $60$   \\
   & $\mathrm{Sym}(3)'$ & $60$   \\
   & $G(5,5,2)$ & $36$   \\
   & $G(2,1,2)$ & $45$    \\
  \hline  
\end{tabular}
\end{center}
\caption{The Conjugacy Classes of the Parabolic Subgroups of $G_{23}, \dots, G_{27}$} \label{Rk3}
\end{table}

\begin{align*}
T_{H_3}(x,y) =\ & y^{12} + 3y^{11} + 6y^{10} + 10y^9 + 15y^8 + 21y^7 + 28y^6 + 36y^5 + 6xy^3 + 45y^4 + x^3 + 12xy^2 \\ & + 49y^3 + 12x^2 + 28xy + 48y^2 + 32x + 32y. 
\end{align*}

\begin{align*}
T_{J_3^{(4)}}(x,y) =\ & y^{18} + 3y^{17} + 6y^{16} + 10y^{15} + 15y^{14} + 21y^{13} + 28y^{12} + 36y^{11} + 45y^{10} + 55y^9 + 66y^8 \\
& + 78y^7 + 91y^6 + 105y^5 + 120y^4 + x^3 + 21xy^2 + 136y^3 + 18x^2 + 70xy + 132y^2 \\ & + 80x + 80y.
\end{align*}

\begin{align*}
T_{L_3}(x,y) =\ & y^9 + 3y^8 + 6y^7 + 10y^6 + 15y^5 + 21y^4 + x^3 + 9xy^2 + 28y^3 + 9x^2 + 18xy + 27y^2 \\
& + 18x + 18y.
\end{align*}

\begin{align*}
T_{M_3}(x,y) =\ & y^{18} + 3y^{17} + 6y^{16} + 10y^{15} + 15y^{14} + 21y^{13} + 28y^{12} + 36y^{11} + 45y^{10} + 55y^9 + 66y^8 \\
& + 78y^7 + 91y^6 + 105y^5 + 12xy^3 + 120y^4 + x^3 + 33xy^2 + 124y^3 + 18x^2 + 54xy \\
& + 108y^2 + 72x + 72y.
\end{align*}

\begin{align*}
T_{J_3^{(5)}}(x,y) =\ & y^{42} + 3y^{41} + 6y^{40} + 10y^{39} + 15y^{38} + 21y^{37} + 28y^{36} + 36y^{35} + 45y^{34} + 55y^{33} + 66y^{32} \\
& + 78y^{31} + 91y^{30} + 105y^{29} + 120y^{28} + 136y^{27} + 153y^{26} + 171y^{25} + 190y^{24} + 210y^{23} \\ 
& + 231y^{22} + 253y^{21} + 276y^{20} + 300y^{19} + 325y^{18} + 351y^{17} + 378y^{16} + 406y^{15} + 435y^{14} \\
& + 465y^{13} + 496y^{12} + 528y^{11} + 561y^{10} + 595y^9 + 630y^8 + 666y^7 + 703y^6  + 741y^5 \\
& + 36xy^3 + 780y^4 + x^3 + 117xy^2 + 784y^3 + 42x^2 + 318xy + 708y^2 + 432x + 432y.
\end{align*}

\section{The Primitive Complex Reflection Groups of Rank $4$}   \label{SecRk4}

\noindent In this section are exposed the Tutte polynomials associated to the primitive CRGs $N_4$, $H_4$, $O_4$, $L_4$. The calculations are done using the cardinalities of conjugacy classes in Table~\ref{N4} -- \ref{L4}, established by means of $\#\mathrm{Cl}_G(X) = \big[G : \mathrm{N}_G(X)\big]$ and \cite[Table~6 -- 9]{KrTa}.

\begin{table}
\begin{center}
\begin{tabular}{|>{\columncolor{Sky}} c | c || >{\columncolor{Sky}} c | c |}
  \hline			
  $P$ & $\#\mathrm{Cl}_{N_4}(P)$ & $P$ & $\#\mathrm{Cl}_{N_4}(P)$  \\
  \hline
  $\mathrm{Sym}(2)$ & $40$ & $\mathrm{Sym}(4)$ & $80$   \\
  $2\mathrm{Sym}(2)$ & $120$ & $\mathrm{Sym}(4)'$ & $80$    \\
  $\mathrm{Sym}(3)$ & $160$ & $G(4,4,3)$ & $20$    \\
  $G(2,1,2)$ & $30$ & $G(2,1,3)$ & $40$    \\
  $\mathrm{Sym}(2) + \mathrm{Sym}(3)$ & $160$ &   &    \\
  \hline  
\end{tabular}
\end{center}
\caption{The Conjugacy Classes of the Parabolic Subgroups of $G_{29}$} \label{N4}
\bigskip
\bigskip
\end{table}

\begin{align*}
T_{N_4}(x,y) =\ & y^{36} + 4y^{35} + 10y^{34} + 20y^{33} + 35y^{32} + 56y^{31} + 84y^{30} + 120y^{29} + 165y^{28} + 220y^{27} \\
& + 286y^{26} + 364y^{25} + 455y^{24} + 560y^{23} + 680y^{22} + 816y^{21} + 969y^{20} + 1140y^{19} \\
& + 1330y^{18} + 1540y^{17} + 1771y^{16} + 2024y^{15} + 2300y^{14} + 2600y^{13} + 2925y^{12} \\
& + 3276y^{11} + 20xy^9 + 3654y^{10} + 60xy^8 + 4040y^9 + 120xy^7 + 4415y^8 + 240xy^6 + 4760y^7 \\
& + 420xy^5 + 5016y^6 + 660xy^4 + 5124y^5 + x^4 + 30x^2y^2 + 1120xy^3 + 5025y^4 + 36x^3 \\
& + 220x^2y + 1740xy^2 + 4500y^3 + 416x^2 + 2240xy + 3360y^2 + 1536x + 1536y.
\end{align*}

\begin{table}
\begin{center}
\begin{tabular}{|>{\columncolor{Sky}} c | c || >{\columncolor{Sky}} c | c |}
  \hline			
  $P$ & $\#\mathrm{Cl}_{H_4}(P)$ & $P$ & $\#\mathrm{Cl}_{H_4}(P)$  \\
  \hline
  $\mathrm{Sym}(2)$ & $60$ & $\mathrm{Sym}(2) + \mathrm{Sym}(3)$ & $600$   \\
  $2\mathrm{Sym}(2)$ & $450$ & $\mathrm{Sym}(2) + G(5,5,2)$ & $360$    \\
  $\mathrm{Sym}(3)$ & $200$ & $\mathrm{Sym}(4)$ & $300$    \\
  $G(5,5,2)$ & $72$ & $H_3$ & $60$    \\
  \hline  
\end{tabular}
\end{center}
\caption{The Conjugacy Classes of the Parabolic Subgroups of $G_{30}$} \label{H4}
\bigskip
\bigskip
\end{table}

\begin{align*}
T_{H_4}(x,y) =\ & y^{56} + 4y^{55} + 10y^{54} + 20y^{53} + 35y^{52} + 56y^{51} + 84y^{50} + 120y^{49} + 165y^{48} + 220y^{47} \\ 
& + 286y^{46} + 364y^{45} + 455y^{44} + 560y^{43} + 680y^{42} + 816y^{41} + 969y^{40} + 1140y^{39} \\
& + 1330y^{38} + 1540y^{37} + 1771y^{36} + 2024y^{35} + 2300y^{34} + 2600y^{33} + 2925y^{32} + 3276y^{31} \\ & + 3654y^{30} + 4060y^{29} + 4495y^{28} + 4960y^{27} + 5456y^{26} + 5984y^{25} + 6545y^{24} \\
& + 7140y^{23} + 7770y^{22} + 8436y^{21} + 9139y^{20} + 9880y^{19} + 10660y^{18} + 11480y^{17} \\
& + 12341y^{16} + 13244y^{15} + 14190y^{14} + 60xy^{12} + 15180y^{13} + 180xy^{11} + 16155y^{12} \\
& + 360xy^{10} + 17056y^{11} + 600xy^9 + 17824y^{10} + 900xy^8 + 18400y^9 + 1260xy^7 \\
& + 18725y^8 + 1680xy^6 + 18740y^7 + 2160xy^5 + 18386y^6 + 72x^2y^3 + 2700xy^4 + 17604y^5 \\
& + x^4 + 144x^2y^2 + 3816xy^3 + 16335y^4 + 56x^3 + 416x^2y + 4932xy^2 + 13932y^3 + 964x^2 \\
& + 6248xy + 10324y^2 + 5040x + 5040y.
\end{align*}

\begin{table}
\begin{center}
\begin{tabular}{|>{\columncolor{Sky}} c | c || >{\columncolor{Sky}} c | c |}
  \hline			
  $P$ & $\#\mathrm{Cl}_{O_4}(P)$ & $P$ & $\#\mathrm{Cl}_{O_4}(P)$  \\
  \hline
  $\mathrm{Sym}(2)$ & $60$ & $\mathrm{Sym}(2) + \mathrm{Sym}(3)$ & $960$   \\
  $2\mathrm{Sym}(2)$ & $360$ &  $\mathrm{Sym}(4)$  & $480$     \\
  $\mathrm{Sym}(3)$ & $320$ &  $G(4,2,3)$  & $60$    \\
  $G(4,2,2)$ & $30$ &    &     \\
  \hline  
\end{tabular}
\end{center}
\caption{The Conjugacy Classes of the Parabolic Subgroups of $G_{31}$} \label{O4}
\end{table}

\begin{align*}
T_{O_4}(x,y) =\ & y^{56} + 4y^{55} + 10y^{54} + 20y^{53} + 35y^{52} + 56y^{51} + 84y^{50} + 120y^{49} + 165y^{48} + 220y^{47} \\   & + 286y^{46} + 364y^{45} + 455y^{44} + 560y^{43} + 680y^{42} + 816y^{41} + 969y^{40} + 1140y^{39} + 1330y^{38} \\
& + 1540y^{37} + 1771y^{36} + 2024y^{35} + 2300y^{34} + 2600y^{33} + 2925y^{32} + 3276y^{31} + 3654y^{30} \\ & + 4060y^{29} + 4495y^{28} + 4960y^{27} + 5456y^{26} + 5984y^{25} + 6545y^{24} + 7140y^{23} + 7770y^{22} \\
& + 8436y^{21} + 9139y^{20} + 9880y^{19} + 10660y^{18} + 11480y^{17} + 12341y^{16} + 13244y^{15} \\
& + 14190y^{14} + 60xy^{12} + 15180y^{13} + 180xy^{11} + 16155y^{12} + 360xy^{10} + 17056y^{11} \\
& + 600xy^9 + 17824y^{10} + 900xy^8 + 18400y^9 + 1260xy^7 + 18725y^8 + 1680xy^6 + 18740y^7 \\
& + 30x^2y^4 + 2160xy^5 + 18386y^6 + 60x^2y^3 + 2640xy^4 + 17604y^5 + x^4 + 90x^2y^2 \\
& + 3480xy^3 + 16365y^4 + 56x^3 + 440x^2y + 4680xy^2 + 14280y^3 + 976x^2 + 6560xy \\
& + 10960y^2 + 5376x + 5376y.
\end{align*}

\begin{table}
\begin{center}
\begin{tabular}{|>{\columncolor{Sky}} c | c |}
  \hline			
  $P$ & $\#\mathrm{Cl}_{L_4}(P)$  \\
  \hline
  $C_3$ & $40$  \\
  $2C_3$ & $240$   \\
  $SL_2(\mathbb{F}_3)$ & $90$  \\
  $C_3 + SL_2(\mathbb{F}_3)$ & $360$  \\
  $L_3$ & $40$   \\
  \hline  
\end{tabular}
\end{center}
\caption{The Conjugacy Classes of the Parabolic Subgroups of $G_{32}$} \label{L4}
\bigskip
\bigskip
\end{table}

\begin{align*}
T_{L_4}(x,y) =\ & y^{36} + 4y^{35} + 10y^{34} + 20y^{33} + 35y^{32} + 56y^{31} + 84y^{30} + 120y^{29} + 165y^{28} + 220y^{27} + 286y^{26} \\
& + 364y^{25} + 455y^{24} + 560y^{23} + 680y^{22} + 816y^{21} + 969y^{20} + 1140y^{19} + 1330y^{18} + 1540y^{17} \\
& + 1771y^{16} + 2024y^{15} + 2300y^{14} + 2600y^{13} + 2925y^{12} + 3276y^{11} + 40xy^9 + 3654y^{10} \\
& + 120xy^8 + 4020y^9 + 240xy^7 + 4335y^8 + 400xy^6 + 4560y^7 + 600xy^5 + 4656y^6 + 840xy^4 \\
& + 4584y^5 + x^4 + 90x^2y^2 + 1120xy^3 + 4305y^4 + 36x^3 + 180x^2y + 1620xy^2 + 3780y^3 \\
& + 396x^2 + 1800xy + 2700y^2 + 1296x + 1296y.
\end{align*}

\section{The Primitive Complex Reflection Groups of Type $K$}  \label{SecRk5}

\noindent In this section are exposed the Tutte polynomials associated to the irreducible CRGs $K_5$, $K_6$. The calculations are done using the cardinalities of conjugacy classes in Table~\ref{K5} -- \ref{K6}, established by means of $\#\mathrm{Cl}_G(X) = \big[G : \mathrm{N}_G(X)\big]$ and \cite[Table~10 -- 11]{KrTa}. Note that the computing of $T_{K_6}(x,y)$ concludes the proof of Theorem~\ref{ThPri}.

\begin{table}
\begin{center}
\begin{tabular}{|>{\columncolor{Sky}} c | c || >{\columncolor{Sky}} c | c |}
  \hline			
  $P$ & $\#\mathrm{Cl}_{K_5}(P)$ & $P$ & $\#\mathrm{Cl}_{K_5}(P)$  \\
  \hline
  $\mathrm{Sym}(2)$  & $45$ &  $G(3,3,3)$ & $40$   \\
  $2\mathrm{Sym}(2)$  & $270$  &  $\mathrm{Sym}(2) + \mathrm{Sym}(4)$  & $540$    \\
  $\mathrm{Sym}(3)$  & $240$  &  $\mathrm{Sym}(5)$  & $216$   \\
  $\mathrm{Sym}(2) + \mathrm{Sym}(3)$  & $720$  &  $G(3,3,4)$  & $40$   \\
  $\mathrm{Sym}(4)$  & $540$  &  $G(2,2,4)$  & $45$   \\
  $3\mathrm{Sym}(2)$  & $270$  &   &  \\
  \hline  
\end{tabular}
\end{center}
\caption{The Conjugacy Classes of the Parabolic Subgroups of $G_{33}$} \label{K5}
\bigskip
\bigskip
\end{table}

\begin{align*}
T_{K_5}(x,y) =\ & y^{40} + 5y^{39} + 15y^{38} + 35y^{37} + 70y^{36} + 126y^{35} + 210y^{34} + 330y^{33} + 495y^{32} + 715y^{31} \\
& + 1001y^{30} + 1365y^{29} + 1820y^{28} + 2380y^{27} + 3060y^{26} + 3876y^{25} + 4845y^{24} + 5985y^{23} \\ & + 7315y^{22} + 8855y^{21} + 10626y^{20} + 12650y^{19} + 14950y^{18} + 17550y^{17} + 20475y^{16} \\
& + 40xy^{14} + 23751y^{15} + 160xy^{13} + 27365y^{14} + 400xy^{12} + 31265y^{13} + 800xy^{11} \\
& + 35360y^{12} + 1400xy^{10} + 39520y^{11} + 2240xy^9 + 43576y^{10} + 3405xy^8 + 47320y^9 \\
& + 40x^2y^6 + 4980xy^7 + 50460y^8 + 120x^2y^5 + 7186xy^6 + 52620y^7 + 240x^2y^4 + 10164xy^5 \\
& + 53164y^6 + x^5 + 940x^2y^3 + 14055xy^4 + 51276y^5 + 40x^4 + 240x^3y + 2220x^2y^2 \\
& + 18460xy^3 + 45960y^4 + 580x^3 + 4080x^2y + 20820xy^2 + 36040y^3 + 3600x^2 + 17856xy \\
& + 22320y^2 + 8064x + 8064y.
\end{align*}

\begin{table}
\begin{center}
\begin{tabular}{|>{\columncolor{Sky}} c | c || >{\columncolor{Sky}} c | c |}
  \hline			
  $P$ & $\#\mathrm{Cl}_{K_6}(P)$ & $P$ & $\#\mathrm{Cl}_{K_6}(P)$  \\
  \hline
  $\mathrm{Sym}(2)$  & $126$ &  $2\mathrm{Sym}(3)$ & $30240$  \\
  $2\mathrm{Sym}(2)$  & $2835$  &  $\mathrm{Sym}(2) + G(3,3,3)$  & $5040$   \\
  $\mathrm{Sym}(3)$  & $1680$  &  $G(2,2,4)$  & $2835$  \\
  $\mathrm{Sym}(2) + \mathrm{Sym}(3)$  & $30240$  &  $\mathrm{Sym}(2) + \mathrm{Sym}(5)$  & $27216$  \\
  $\mathrm{Sym}(4)$  & $11340$  &  $\mathrm{Sym}(3) + \mathrm{Sym}(4)$  & $45360$   \\
  $3\mathrm{Sym}(2)$  & $11340$  &  $\mathrm{Sym}(2) + G(3,3,4)$  & $5040$  \\
  $G(3,3,3)$  & $560$  &  $\mathrm{Sym}(6)$  & $9072$  \\
  $\mathrm{Sym}(2) + \mathrm{Sym}(4)$  & $68040$  &  $\mathrm{Sym}(6)'$  & $9072$  \\
  $G(3,3,4)$  & $1680$  &  $G(2,2,5)$  & $3402$  \\
  $\mathrm{Sym}(5)$  & $27216$  &  $G(3,3,5)$  & $672$  \\
  $2\mathrm{Sym}(2) + \mathrm{Sym}(3)$  & $45360$  &  $K_5$  & $126$   \\
  \hline  
\end{tabular}
\end{center}
\caption{The Conjugacy Classes of the Parabolic Subgroups of $G_{34}$} \label{K6}
\end{table}

\clearpage

\begin{align*}
T_{K_6}(x,y) =\ & y^{120} + 6y^{119} + 21y^{118} + 56y^{117} + 126y^{116} + 252y^{115} + 462y^{114} + 792y^{113} + 1287y^{112} \\
& + 2002y^{111} + 3003y^{110} + 4368y^{109} + 6188y^{108} + 8568y^{107} + 11628y^{106} + 15504y^{105} \\
& + 20349y^{104} + 26334y^{103} + 33649y^{102} + 42504y^{101} + 53130y^{100} + 65780y^{99} + 80730y^{98}\\
& + 98280y^{97} + 118755y^{96} + 142506y^{95} + 169911y^{94} + 201376y^{93} + 237336y^{92} \\
& + 278256y^{91} + 324632y^{90} + 376992y^{89} + 435897y^{88} + 501942y^{87} + 575757y^{86} \\
& + 658008y^{85} + 749398y^{84} + 850668y^{83} + 962598y^{82} + 1086008y^{81} + 1221759y^{80} \\
& + 1370754y^{79} + 1533939y^{78} + 1712304y^{77} + 1906884y^{76} + 2118760y^{75} + 2349060y^{74} \\
& + 2598960y^{73} + 2869685y^{72} + 3162510y^{71} + 3478761y^{70} + 3819816y^{69} + 4187106y^{68} \\
& + 4582116y^{67} + 5006386y^{66} + 5461512y^{65} + 5949147y^{64} + 6471002y^{63} + 7028847y^{62} \\
& + 7624512y^{61} + 8259888y^{60} + 8936928y^{59} + 9657648y^{58} + 10424128y^{57} + 11238513y^{56} \\
& + 12103014y^{55} + 13019909y^{54} + 13991544y^{53} + 15020334y^{52} + 16108764y^{51} \\
& + 17259390y^{50} + 18474840y^{49} + 19757815y^{48} + 21111090y^{47} + 22537515y^{46} \\
& + 24040016y^{45} + 25621596y^{44} + 27285336y^{43} + 29034396y^{42} + 126xy^{40} + 30872016y^{41} \\
& + 630xy^{39} + 32801391y^{40} + 1890xy^{38} + 34825546y^{39} + 4410xy^{37} + 36947211y^{38} \\
& + 8820xy^{36} + 39168696y^{37} + 15876xy^{35} + 41491766y^{36} + 26460xy^{34} + 43917516y^{35} \\
& + 41580xy^{33} + 46446246y^{34} + 62370xy^{32} + 49077336y^{33} + 90090xy^{31} + 51809121y^{32} \\
& + 126126xy^{30} + 54638766y^{31} + 171990xy^{29} + 57562141y^{30} + 229320xy^{28} + 60573696y^{29} \\
& + 299880xy^{27} + 63666336y^{28} + 385560xy^{26} + 66831296y^{27} + 489048xy^{25} + 70058016y^{26} \\
& + 613830xy^{24} + 73333344y^{25} + 764190xy^{23} + 76640739y^{24} + 945210xy^{22} + 79959474y^{23} \\
& + 1162770xy^{21} + 83263839y^{22} + 1423548xy^{20} + 86522344y^{21} + 1735020xy^{19} \\
& + 89696922y^{20} + 2105460xy^{18} + 92742132y^{19} + 2543940xy^{17} + 95604362y^{18} \\
& + 3060330xy^{16} + 98221032y^{17} + 1680x^2y^{14} + 3668700xy^{15} + 100519797y^{16} \\
& + 6720x^2y^{13} + 4389000xy^{14} + 102414348y^{15} + 16800x^2y^{12} + 5236980xy^{13} \\
& + 103796853y^{14} + 33600x^2y^{11} + 6224190xy^{12} + 104540478y^{13} + 58800x^2y^{10} \\
& + 7357980xy^{11} + 104501908y^{12} + 94080x^2y^9 + 8659644xy^{10} + 103523868y^{11} \\
& + 143955x^2y^8 + 10164420xy^9 + 101419500y^{10} + 560x^3y^6 + 212940x^2y^7 + 11915820xy^8 \\
& + 97956740y^9 + 1680x^3y^5 + 336126x^2y^6 + 13948620xy^7 + 92845530y^8 + 3360x^3y^4 \\
& + 538524x^2y^5 + 16253244xy^6 + 85742040y^7 + x^6 + 16940x^3y^3 + 845145x^2y^4 \\
& + 18547956xy^5 + 76257370y^6 + 120x^5 + 1680x^4y + 42420x^3y^2 + 1315020x^2y^3 \\
& + 20364540xy^4 + 64204140y^5 + 5580x^4 + 103320x^3y + 1787940x^2y^2 + 20946240xy^3 \\
& + 49757400y^4 + 125280x^3 + 2068416x^2y + 19005840xy^2 + 33672240y^3 + 1353024x^2 \\
& + 13716864xy + 17962560y^2 + 5598720x + 5598720y.
\end{align*}

\bibliographystyle{abbrvnat}

\end{document}